\documentclass[preprint,12pt]{elsarticle}
\usepackage{hyperref}
\usepackage{amssymb}
\usepackage{latexsym}
\usepackage{amsmath,amsfonts,amssymb}
\usepackage{graphicx}
\usepackage{color}
\newenvironment{proof}{\medskip                    
\noindent{\scshape Proof:}}{\quad $\square$
\medskip}  

\newtheorem{theorem}{Theorem}[section]
\newtheorem{lemma}[theorem]{Lemma}
\newtheorem{proposition}[theorem]{Proposition}
\newtheorem{corollary}[theorem]{Corollary}
\newtheorem{remark}[theorem]{Remark}

\newtheorem{example}[theorem]{Example}
\newtheorem{definition}[theorem]{Definition}

\newcommand{\thmref}[1]{Theorem~\ref{#1}}

\newcommand{\cW}{{\cal W}}

\newcommand{\R}{{\mathbb R}}
\newcommand{\Rp}{{\mathbb R}_+}
\newcommand{\Rpm}{{\mathbb R}_+^m}
\newcommand{\Rpn}{{\mathbb R}_+^n}
\newcommand{\Rpnn}{{\mathbb R}_+^{n\times n}}
\newcommand{\Rpmn}{{\mathbb R}_+^{m\times n}}

\newcommand{\cK}{{\mathcal K}}
\newcommand{\crit}{{\mathcal G}_c}
\newcommand{\digr}{{\mathcal G}}

\newcommand{\SisV}{\operatorname{Sis}V}

\def\attr{\mathop{\rm attr}}
\def\kd{\hfil$\square$\linebreak}
\newcommand{\supp}{\operatorname{supp}}

\newcommand{\Sat}{\operatorname{Sat}}
\newcommand{\spann}{\operatorname{span}_{\oplus}}
\newcommand{\diag}{\operatorname{diag}}

\newcommand{\Attr}{\operatorname{attr}}

\newcommand{\RR}{{\mathbb{R}}}

\def\mbf#1{\mbox{\boldmath$#1$}}
\begin{document}

%
\setcounter{footnote}{1}

\begin{frontmatter}
\setcounter{footnote}{1}
\title{$\mbf{X}$-simple image eigencones of tropical matrices}
\author[rvt]{J\'an Plavka\fnref{fn1} }
\ead{jan.plavka@tuke.sk}

\author[rvt2]{Serge{\u\i} Sergeev\corref{cor}}
\ead{sergiej@gmail.com}

\address[rvt]{Department of Mathematics and Theoretical Informatics,  Technical  University,\\
B. N\v emcovej 32, 04200 Ko\v sice, Slovakia}

\address[rvt2]{University of Birmingham, School of Mathematics, Edgbaston B15 2TT, UK}

\cortext[cor]{Corresponding author. Email: sergiej@gmail.com}
\fntext[fn1]{Supported by   APVV grant 04-04-12.}

\begin{abstract}

We investigate max-algebraic (tropical) one-sided systems $A\otimes x=b$ where $b$ is an eigenvector and 
$x$ lies in an interval $\mbf{X}$. A matrix $A$ is said to have $\mbf{X}$-simple image eigencone associated with an eigenvalue $\lambda$, if any eigenvector $x$ associated with $\lambda$ and belonging to the interval $\mbf{X}$ is the unique solution of the system
$A\otimes y=\lambda x$ in $\mbf{X}$. We characterize matrices with
$\mbf{X}$-simple image eigencone geometrically and combinatorially, and 
for some special cases, derive criteria that can be efficiently checked in 
practice.

\end{abstract}
\end{frontmatter}



\section{Introduction}

\subsection{Problem statement and main results}

In this paper, by max algebra we mean the set of nonnegative numbers $\Rp$ equipped with usual multiplication $a\otimes b:=a\cdot b$
and idempotent addition $a\oplus b:=\max(a,b)$. Algebraically, $\Rp$ equipped with these operations forms a semifield. 
The operations of max algebra are then extended
to matrices and vectors in the usual way, giving rise to an analogue of nonnegative linear algebra.
 
Max-algebraic one-sided systems $A\otimes x=b$ and max-algebraic eigenproblem
$A\otimes x=\lambda x$ are two fundamental problems of max algebra whose solution
goes back to the works of Cuninghame-Green~\cite{CG:79,CG:95}, Vorobyev~\cite{Vor} and Zimmermann~\cite{Zim} and
these two topics are thoroughlly discussed in any textbook of the max-plus (tropical) linear algebra~\cite{BCOQ,But,HOW}. 
Our intention is to consider the situation when the right-hand side of $A\otimes x=b$
is an eigenvector, and also when the solution has to lie in some interval of $\Rpn$.

By an interval of $\Rpn$ we mean a subset of $\Rpn$ of the form
$\mbf{X}=\times_{i=1}^n \mbf{X}_i$, where each $\mbf{X}_i$ is an arbitrary interval 
belonging to $\Rp$, with its upper end possibly equal to $+\infty$ (and lower
end possibly equal to $0$). In particular, $\Rpn$ is an interval of itself.
For each $i$ we denote $\underline{x}_i:=\inf \mbf{X}_i$ and $\overline{x}_i=\sup\mbf{X}_i$. Then we also have 
$\underline{x}:=(\underline{x}_i)_{i=1}^n=\inf\mbf{X}$ and $\overline{x}:=(\overline{x}_i)_{i=1}^n=\sup\mbf{X}$.

The notion of $\mbf{X}$-simple image eigencone, which we introduce next, is
related to the concept of simple image set~\cite{ButSIS}.
By definition, simple image set of $A$ is the set of vectors $b$ such that the system
$A\otimes x=b$ has a unique solution.  If the only solution of the system $A\otimes x=b$ is $x=b$, then $b$ is
called a {\em simple image eigenvector.} 

A matrix $A$ is said to have $\mbf{X}$-simple image eigencone associated with a (fixed) eigenvalue 
$\lambda$, if any eigenvector $x$ associated with the eigenvalue $\lambda$ and belonging to the interval $\mbf{X}$ is the unique solution in $\mbf{X}$ for the system
$A\otimes y=\lambda x$. The characterization of a matrix with  $\mbf{X}$-simple image eigencone is described as the main result of the paper in  Section~\ref{s:int}.

Let us now give more details on the organization of the paper and on the results obtained there. Section~\ref{s:prel} is devoted to basic notions of max algebra and its connections to the theory of digraphs and max-algebraic (tropical) convexity. In particular, we revisit here the spectral theory, focusing on the eigencone associated with an arbitrary eigenvalue, the generating matrix and the critical graph. Some aspects of the diagonal similarity scaling are also briefly discussed. 

Section~\ref{s:sis} starts by discussing the problem of covering the node set of a digraph by ingoing edges. We
proceed with the theory of one-sided systems $A\otimes x=b$ where we describe the solution set to 
such systems and start analysing the case when $b$ is an eigenvector of $A$. The main result of that section is 
Theorem~\ref{t:peter}, which characterizes matrices that have at least one simple image eigenvector corresponding
to an eigenvalue $\lambda$.  

In the beginning of Section~\ref{s:int} we first discuss the relation between $\mbf{X}$-simple image 
eigencone and $\mbf{X}$-weak robustness of a matrix. We then develop an interval version of 
theory of one-sided systems $A\otimes x=b$, i.e., when $x$ has to belong to an interval $\mbf{X}$.
The second part of the section contains the main results of the paper: Theorem~\ref{t:mainres} and 
Theorem~\ref{t:openclosed}. More precisely, Theorem~\ref{t:mainres} characterizes when $A$ has $\mbf{X}$-simple 
image eigencone in general, and Theorem~\ref{t:openclosed} focuses on the case when
$\mbf{X}$ is of a certain special type. 

In the end of the paper we formulate some conclusions and discuss some directions for further research.

\subsection{Motivations}

In the literature, max algebra often appears as max-plus semiring developed over the set $\R\cup\{-\infty\}$
equipped with operations $a\otimes b:=a+b$ and $a\oplus b:=\max(a,b)$.   However, this semiring is isomorphic to the semiring defined above, via a logarithmic transform. Max-plus algebra plays the crucial role in the study of discrete-event dynamic  systems connected with the optimization problems such as scheduling or project management in which the objective function depends on the  operations maximum and plus. The main principle of  discrete-event dynamic systems consisting of $n$ entities (machines~\cite{CG:79,CG:95}, processors~\cite{BSSws}, computers, etc.) is that the entities work interactively, i.e., a
given entity must wait before proceeding to its next event
until certain others have completed their current events. Cuninghame-Green~\cite{CG:79} and Butkovi\v{c}~\cite{But} discussed 
 a hypothetical industrial
discrete-event dynamic system and a multiprocessor interactive system, respectively, which can be 
described by the interferences using recurrence relations\\
$x_i(r+1)=\max(x_1(r)+a_{1i},\ x_2(r)+a_{2i},\dots,\ x_n(r)+a_{ni}),\ i\in\{1,2,\dots,n\}.$
The formula expresses the fact that entity
$i$ must wait with its $r+1$st cycle until entities $j=1,\dots,n$  have   finished their $r$th cycle. The symbol $x_i(r)$
denotes the starting time of the $r$th cycle of entity $i$, and
$a_{ij}$ is the corresponding activity duration at which entity $e_j$ prepares the
outputs (products, components, data, atc.) for entity $e_i$. 
The steady states of such systems correspond to eigenvectors of max-plus matrices, therefore the investigation of properties of eigenvectors is important for the above mentioned applications.

In max-plus algebra the matrices for which the steady states of the systems are reached with any
nontrivial starting vector are called robust. Such matrices have been studied in \cite{But}, \cite{Ser-11}. The matrices for which the steady states of the systems are reached  only if a nontrivial starting vector is an eigenvector of the  matrix are called weakly robust. Efficient characterizations of such 
matrices are described in \cite{BSS}.

In practice, the values of starting vector are not exact numbers and usually they are rather contained in some intervals. Considering matrices and vectors with interval entries is therefore of practical importance.
See~\cite{gazi08,mys05,mys06,p,ro} for some of the recent developments.  
In particular, the weak robustness of an interval matrix is studied in \cite{P1}.

The aim of this paper is to characterize the weak $\mbf{X}$-robustness, i.e., the weak robustness of matrices with initial times confined in an interval vector $\mbf{X}$,  using  $\mbf{X}$-simplicity of image eigencone.


\section{Preliminaries}
\label{s:prel}

\subsection{Matrices and graphs}

Many problems of max algebra can be described and resolved in terms of digraphs (i.e., directed graphs). Let us
give some of the relevant definitions here.

\if{
\begin{definition}[Associated digraphs]
\label{def:digraphs}
{\rm For a matrix $A\in \Rpnn$ and $N=\{1,\ldots,n\}$ the symbol $\digr(A)=(N,E)$ stands for the 
{\em weighted digraph associated with $A$:} the node set of $\digr(A)$ 
is $N$, the edge and the weight of any arc $(i,j)$ is $a_{ij}$. }
\end{definition}
}\fi

\begin{definition}[Associated digraphs]
\label{def:digraphs}
{\rm The {\em weighted digraph associated with} $A\in\Rpnn$ is the digraph
$\digr(A)=(N,E)$ with the node set $N:=\{1,\ldots,n\}$ and the edge set $E$ such that
$(i,j)\in E$ (edge from $i$ to $j$) if and only if $a_{ij}>0$. The number $a_{ij}$ is called the {\em weight} of
$(i,j)$. }
\end{definition}

\begin{definition}[Paths]
\label{def:paths}
{\rm A {\em path} in the digraph $\digr(A)=(N,E)$ is a sequence of nodes
$p=(i_1,\,\ldots,\,i_{k+1})$ such that $(i_j,i_{j+1})\in E$ for
$j=1,\,\ldots,\,k$. A path $p$ is closed if  $i_1 = i_{k+1}$, elementary if all nodes are distinct, and a cycle if it is closed and elementary. The number $k$ is the {\em length} of the path $p$ and
is denoted by $l(p)$. }
\end{definition}


\begin{definition}[Strongly connected components]
\label{def:scc}
{\rm By a \textit{strongly connected component}  (for brevity s.c.c.) of $\digr(A)=(N,E)$ we mean a
subdigraph $\digr'=(N',E')$ with $N'\subseteq N$ and $E'\subseteq E$, 
such that any two distinct nodes $i,j\in N'$ are contained in a
common cycle 
and $N'$ is a maximal subset of $N$ with that property. Particularly, $\digr(A) $ is strongly connected  if $\digr'=(N',E')$ is strongly connected component of $\digr(A)$ with  $N'=N$ and $E'=E$.}
\end{definition}

Powers of max algebraic matrices are closely related to 
optimization on digraphs. Observe that the $i,j$th entry of 
the power $A^{\otimes k}:=\underbrace{A\otimes\ldots\otimes A}_k$ is the 
biggest weight among all paths of length $k$ connecting 
$i$ to $j$.

If we define the formal series $A^+=\bigoplus_{k=1}^{\infty} A^{\otimes k}$ then the $i,j$th entry of $A^+$ (possibly diverging to $+\infty$) equals to the greatest weight among all paths connecting $i$ to $j$. Such weight is guaranteed to be finite if the weight of any cycle in 
$\digr(A)$ does not exceed $1$.


\begin{definition}[Irreducibility]
{\rm A matrix $A\in\Rpnn$  is called {\em irreducible} if 
$\digr(A)$ is strongly connected, and {\em reducible} otherwise.
}
\end {definition}

\begin{definition}[Graph restrictions]
{\rm For arbitrary $K\subseteq N$, we denote by $\digr(A)|_{K}$ the subgraph of $\digr(A)$ consisting of all nodes of 
$K$ and all edges of $\digr(A)$ between the nodes of $K$.}
\end{definition}

\subsection{Geometry}

Max algebra also gives rise to the max-algebraic (tropical) analogue of convexity.

\begin{definition}[Max cone]
\label{def:maxcone}
{\rm A subset $\cK\subseteq\Rpn$ is called a {\em max cone} if we have\\
1) $\lambda x\in \cK$ for any $\lambda\geq 0$ and $x\in \cK$,\\
2) $x\oplus y\in \cK$ for any $x,y\in\cK$.
}
\end{definition}

The name ``max cone'' was suggested in~\cite{BSS}. In the literature this object also appears
as tropical cone or max-plus linear space. 

\begin{definition}[Column span]
\label{def:colspan}
{\rm For $A\in\Rpmn$ define its {\em max-algebraic column span} as
\begin{equation}
\spann(A):=\left\{\bigoplus_{j=1}^n A_{\cdot j}x_j\colon x_j\geq 0\;\forall j\right\}, 
\end{equation}
where $A_{\cdot j}$, for $j=1,\ldots,n$ denotes the $j$-th column of $A$.\\
The set of all positive vectors in $\spann(A)$ will be denoted by $\spann^+(A)$.
}
\end{definition}

It is easily shown that $\spann(A)$ is a max cone. 
Furthermore, $\spann(A)$ is always closed in the 
Euclidean topology~\cite{BSS}.

Consider now the following operator.

\begin{definition}[Projector]
\label{def:proj}
{\rm Let $\cW$ be a closed max cone. Define
\begin{equation}
\label{e:projector}
P_{\cW}(x):=\max\{y\in\cW\colon y\leq x\}.
\end{equation}
}
\end{definition}

In the case $\cW=\spann(A)$ we will write $P_A$ instead of $P_{\spann(A)}$, for brevity.

$P_{\cW}$ is a nonlinear projector on the max cone $\cW$. 
It is  homogeneous ($P_{\cW}(\lambda x)=\lambda P_{\cW} x$) and isotone
($x\leq y\Rightarrow P_{\cW} x\leq P_{\cW} y$.)  
These operators are crucial for tropical convexity: see, e.g.,~\cite{CGQS}.

\subsection{Eigenvalues and eigenvectors}

\begin{definition}[Eigencone]
\label{def:eigencone}
{\rm The set 
\begin{equation}
V(A,\lambda)=\{x\colon A\otimes x=\lambda x\},
\end{equation}
where $A\in\Rpnn$ and $\lambda\geq 0$, is called 
the (max-algebraic) {\em eigencone} of $A$ associated with $\lambda$.
The nonzero vectors of $V(A,\lambda)$ are
(max-algebraic) {\em eigenvectors} of $A$ associated with $\lambda$.\\
The set of all positive vectors in $V(A,\lambda)$ will be denoted by 
$V^+(A,\lambda)$.}
\end{definition}

Note that $V(A,\lambda)$ consists of the eigenvectors associated with $\lambda$ and vector ${\mathbf 0}$. 
Obviously, $V(A,\lambda)$ is a max cone.

\begin{definition}[Maximum cycle geometric mean]
\label{def:mcgm}
{\rm Let $A=(a_{ij})\in\Rpnn$. For any  $i_1,\ldots,i_k\in N$, the {\em geometric mean} of the
cycle $(i_1,i_2\ldots, i_k)$ is defined as $\sqrt[k]{a_{i_1i_2}\cdot\ldots a_{i_ki_1}}$.   
The {\em maximum cycle geometric mean} of $A\in\Rpnn$ equals to}
\begin{equation}
\label{e:lambda}
\lambda(A):=
\max_{k=1}^n\max_{1\leq i_1,\ldots,i_k\leq n} \sqrt[k]{a_{i_1i_2}\cdot\ldots a_{i_ki_1}}.
\end{equation}
\end{definition} 

$\lambda(A)$ is the greatest 
max-algebraic eigenvalue of $A$, for any $A\in\Rpnn$.
 If $A$ is irreducible then $\lambda(A)$ is the only max-algebraic eigenvalue of $A$(e.g.~\cite{But}, Theorem 4.4.8).

Reducible $A\in\Rpnn$ may have up to $n$ max-algebraic eigenvalues, in general. We next give some elements 
of the spectral theory of reducible matrices. Although that theory is usually developed in terms of the Frobenius 
normal form and spectral classes~\cite{But} following a similar development of the spectral theory of 
nonnegative matrices, we choose not to do so, since in this paper we only need 1) the relation between the
critical graph and the saturation graph, 2) the generating matrix of $V(A,\lambda)$.

Denote by $\Lambda(A)$ the set of (max-algebraic) eigenvalues of $A$. General $\lambda\in\Lambda(A)$
can be characterized as maximum cycle geometric mean of a certain subgraph of $\digr(A)$. 

\begin{definition}[Support]
\label{def:support}
For $x\in\Rpn$ the set $\supp(x)=\{i\colon x_i>0\}$ is called the {\em support} of $x$.
\end{definition}

The proof of the following statement is standard, but we give it for the reader's convenience.

\begin{proposition}
\label{p:mcgm}
Let $\lambda\in\Lambda(A)$ and $x\in V(A,\lambda)$ be nonzero. Then $\lambda$ is the maximum
cycle geometric mean of $\digr(A)|_{\supp(x)}$.
\end{proposition}
\begin{proof}
Take any cycle $(i_1,\ldots,i_k)$ with all indices belonging to $\supp(x)$. 
Multiplying the inequalities $a_{i_li_{l+1}} x_{i_{l+1}}\leq \lambda x_{i_l}$ for $l=1,\ldots,k-1$
and $a_{i_ki_1}x_{i_1}\leq \lambda x_{i_k}$, and cancelling all the coordinates of $x$ we obtain that
the cycle mean of $(i_1,\ldots,i_k)$ does not exceed $\lambda$.

Now, start with any $j_0\in\supp(x)$ and find $j_1$ such that $a_{j_0j_1} x_{j_1}=\lambda x_{j_0}$.
We again have $j_1\in\supp(x)$. Proceeding this way we obtain a cycle $(j_t,\ldots j_{t+l})$ (for some $t$ and $l$)
with the cycle mean equal to $\lambda$. 
\end{proof}

Let us also recall a useful link between the 
support of an eigenvector of $A$ and the zero-nonzero pattern of $A$.

Let $A\in\Rpnn$ and $J,L\subseteq N$. $A_{J,L}$ denotes the submatrix of $A$ with row index set $J$ and column index set $L$.

\begin{proposition}[e.g. \cite{But}, p.96 ]
\label{p:support}
Let $A\in\Rpnn$,  $x\in V(A,\lambda)$
and $N':=\supp(x)$. Then $A_{N\backslash N', N'}={\textbf{0}}$.
\end{proposition}

\begin{definition}[Critical graph $\crit(A,x,\lambda)$]
\label{def:critgr1}
For $x\in V(A,\lambda)$, define the critical graph $\crit(A,x,\lambda)$ as the 
subgraph of $\digr(A)$ consisting of all nodes and edges belonging to the cycles of $\digr(A)|_{\supp(x)}$
whose geometric mean is equal to $\lambda$.
\end{definition}

\begin{definition}[Saturation graph]
\label{def:satgr}
{\rm For $x\in V(A,\lambda)$, the {\em saturation graph} $\Sat(A,x,\lambda)$ is the subgraph of $\digr(A)$ with
 set of nodes $N$ and set of edges}
\begin{equation}
\label{e:satgraph}
E_{\Sat}=\{(i,j)\colon a_{ij}x_j=\lambda x_i\neq 0\}.
\end{equation}
\end{definition}

\begin{proposition}[\cite{BCOQ}, Theorems 3.96 and 3.98]
\label{p:satgraph}
For any $x\in V(A,\lambda)$, 
\begin{itemize}
\item[(i)] Every node $i\in N$ such that $x_i\neq 0$ 
has an outgoing edge in $\Sat(A,x,\lambda)$;
\item[(ii)] Any cycle 
in $\Sat(A,x,\lambda)$ belongs to $\crit(A,x,\lambda)$;
\item[(iii)] $\crit(A,x,\lambda)$ is a subgraph of 
$\Sat(A,x,\lambda)$;
\end{itemize}
\end{proposition}

\begin{definition}[Critical graph $\crit(A,\lambda)$]
\label{def:critgr2}
The critical graph $\crit(A,\lambda)$ associated with $\lambda$
and the set of nodes $N^{\lambda}$ associated with 
$\lambda$:
\begin{equation}
\label{e:critgr2}
\begin{split}
&\crit(A,\lambda):=\crit(A,x',\lambda),\   N^{\lambda}:=\supp(x'),\ \text{where}\\
&\text{for}\ x'\in V(A,\lambda)\ \text{such that}\ \supp(x)\subseteq\supp(x')\,
\forall x\in V(A,\lambda).
\end{split}
\end{equation}
\end{definition}
 
Each $\crit(A,x,\lambda)$ consists of several strongly connected components isolated from each other.

In the following proposition let us collect some facts about the relation between $\crit(A,x,\lambda)$ and $\crit(A,\lambda)$.

\begin{proposition}
\label{p:critAxl}
Let $A\in\Rpnn$, $\lambda\in\Lambda(A)$ and $x\in V(A,\lambda)$. Then
\begin{itemize}
\item[(i)] $\crit(A,x,\lambda)\subseteq\crit(A,\lambda)$;
\item[(ii)] $\crit(A,x,\lambda)=\crit(A,\lambda)|_{\supp(x)}$;
\item[(iii)] $\crit(A,x,\lambda)$ consists of entire strongly connected components of $\crit(A,\lambda)$.
\end{itemize}
\end{proposition}
\begin{proof}
We only prove (ii) and (iii) since (i) follows from any of them.

\medskip\noindent (ii): Observe that both $\crit(A,x,\lambda)$ and $\crit(A,\lambda)|_{\supp(x)}$ consist of the nodes and edges of the cycles on $\supp(x)$ that have the cycle geometric mean $\lambda$, hence $\crit(A,x,\lambda)=\crit(A,\lambda)|_{\supp(x)}$. 

\medskip\noindent (iii): Since $\crit(A,x,\lambda)$ is defined as a subgraph
consisting of all nodes and edges of some critical cycles, it consists of several isolated strongly connected components, and each of these components
is a subgraph of a component of $\crit(A,\lambda)$. It remains to prove that
none of these subgraphs is proper. 

By the contrary, suppose that one of these components is a 
proper subgraph of a component of $\crit(A,\lambda)$. Since 
$\crit(A,x,\lambda)=\crit(A,\lambda)|_{\supp(x)}$, the component of $\crit(A,\lambda)$ should contain a node in $\supp(x)$ 
and a node not in $\supp(x)$, otherwise
it coincides with the component of $\crit(A,x\lambda)$.
 However, this contradicts with Proposition~\ref{p:support}. 
Hence the claim.
\end{proof}

We further define a generating matrix of $V(A,\lambda)$. 
For that, first define the matrix $A_{\lambda}$ with the columns
\begin{equation}
\label{e:alambda}
(A_{\lambda})_{\cdot i}=
\begin{cases}
A_{\cdot i}/\lambda & \text{if $i\in N^{\lambda}$},\\
{\mathbf 0}, & \text{otherwise}.
\end{cases}
\end{equation}
Here $N^{\lambda}$ is as in~\eqref{e:critgr2}.
$A_{\lambda}^+$ is finite, since the weight of any cycle in $A_{\lambda}$ does not exceed $1$.

\begin{definition}[Generating Matrix]
\label{def:GA}
Let $N_c^{\lambda,1},\ldots,N_c^{\lambda,k}$ be the node sets of the 
strongly connected components of $\crit(A,\lambda)$, and let $j_1,\ldots,j_k$ be the first 
indices in those components. Define the generating matrix of $V(A,\lambda)$ as 
the matrix resulting from stacking the columns  $(A_{\lambda}^+)_{\cdot j_1},\ldots,
(A_{\lambda}^+)_{\cdot j_k}$ together:
\begin{equation}
\label{e:galambda}
G_{A,\lambda}=[(A_{\lambda}^+)_{\cdot j_1},\ldots,
(A_{\lambda}^+)_{\cdot j_k}]
\end{equation} 
\end{definition}


\begin{proposition}[\cite{But} Coro. 4.6.2]
\label{p:generate}
For any nonzero $\lambda\in\Lambda(A)$
\begin{equation}
\label{e:eigspace}
V(A,\lambda)=V(A_{\lambda},1)=\spann(G_{A,\lambda}).
\end{equation}
\end{proposition}

\if{
\begin{itemize}
\item[{\rm (i)}] If $x\in V(A,\lambda)$, then no node with index in $N\backslash\supp(x)$
has access to any node with index in $\supp(x)$, in $\digr(A)$. 
\item[{\rm (ii)}] All rows of $G_{A,\lambda}$ with indices not in $N^{\lambda}$ are $0$.
\end{itemize}
\end{proposition} 
}\fi

\subsection{Invertible matrices and diagonal similarity scaling}

The class of invertible matrices in max algebra is quite thin. In fact it coincides with that
in nonnegative algebra, consisting of all products of positive diagonal and permutation 
matrices. The positive diagonal matrices will be especially interesting to us
since they give rise to a particularly useful {\em visualization scaling}, also known as 
a Fiedler-Pt\'ak scaling. For a positive vector $x\in\Rpn$, denote by $\diag(x)$ matrix $X\in\Rpnn$ whose
$i$th diagonal entry is $x_i$ and all off-diagonal entries are $0$.

\begin{proposition}[\cite{SSB}, Theorem 3.7]
\label{p:SSB}
Let $A\in\Rpnn$. For a positive $x\in\Rpn$, let $X=\diag(x)$
and $\Tilde{A}=X^{-1}AX$ with entries $\Tilde{a}_{ij}$ for $i,j=1,\ldots,n$.
\begin{itemize}
\item[{\rm (i)}]  If $x$ satisfies $A\otimes x\leq x$ then $\Tilde{a}_{ij}\leq 1$ and, in particular,
$\Tilde{a}_{ij}=1$ for $(i,j)\in E_c(A,1)$ {\bf (visualization scaling)} .
\item[{\rm (ii)}] There exists a positive $x$ satisfying $A\otimes x\leq x$ such that $\Tilde{a}_{ij}\leq 1$, and $\Tilde{a}_{ij}=1$ if and only if $(i,j)\in E_c(A,1)$ {\bf (strict visualization scaling)}.
\end{itemize}
\end{proposition}

\begin{definition}[Visualization]
\label{def:vis}
A matrix $A=(a_{ij})\in\Rpnn$ is called visualized if $a_{ij}\leq \lambda(A)$ for all $i,j\in N$ and 
strictly visualized if it is visualized and $a_{ij}=\lambda(A)$ holds only for $(i,j)\in\crit(A,\lambda)$.
\end{definition}

 We will also use the following observation about the diagonal similarity scaling, where by $\SisV(A,\lambda)$
we denote the set of vectors in $V(A,\lambda)$ that belong to the simple image set of $A$.
\begin{lemma}
\label{l:scale-siseig}
Let $A\in\Rpnn$, let $X\in\Rpnn$ be a positive diagonal matrix. 
and $\Tilde{A}:=X^{-1}AX$. Then 
\begin{itemize}
\item[{\rm (i)}] $\Lambda(A)=\Lambda(\Tilde{A})$, and $\crit(A,\lambda)=\crit(\Tilde{A},\lambda)$ 
for every $\lambda\in\Lambda(A)$;
\item[{\rm (ii)}] $y\in V(A,\lambda)\Leftrightarrow X^{-1}y\in V(\Tilde{A},\lambda)$;
\item[{\rm (iii)}] $y\in\SisV(A,\lambda)\Leftrightarrow X^{-1}y\in\SisV(\Tilde{A},\lambda)$.
\end{itemize}
\end{lemma}
\begin{proof}
The facts described in part (i) are well-known.
Parts (ii) and (iii) follow from the observation that 
$$A\otimes y=b\Leftrightarrow X^{-1}A\otimes y=X^{-1}b\Leftrightarrow X^{-1}AX\otimes (X^{-1}y)=X^{-1}b$$
for all $y,b\in\Rpnn$.
\end{proof}

Let us recall some properties of $A_{\lambda}^+$ and $G_{A,\lambda}$ when $A$ is visualized.
By $E$ we denote a matrix of the same dimensions as $A$ whose every element is $1$.

\begin{proposition}[\cite{SSB}, Proposition 4.1]
\label{p:a+props}
Let $A\in\Rpnn$ be visualized and let $\lambda=\lambda(A)$.
Let $N_c^1,\ldots,N_c^k$ be the node sets of the components of $\crit(A,\lambda)$, and let $K=\{j_1,\ldots,j_k\}$ be the index set of 
the columns of $A_{\lambda}^+$ forming $G_{A,\lambda}$.
\begin{itemize}
\item[{\rm (i)}] For each $r,s\in K$ there exists $\alpha_{rs}\leq 1$ such that 
$(A_{\lambda}^+)_{N_c^rN_c^s}=\alpha_{rs} E_{N_c^rN_c^s}$ and
$(G_{A,\lambda})_{N_c^r,s}=\alpha_{rs} E_{N_c^r,j_s}$. 
\item [{\rm (ii)}] If $r=s$ then $\alpha_{rs}=1$.
\item[{\rm (iii)}] There exists $s$ such that $\alpha_{rs}<1$ for all $r\neq s$.
\end{itemize}
\end{proposition}
\begin{proof}
Parts (i) and (ii) follow from~\cite{SSB} Proposition 4.1, part 2.  For part (iii), note that if it does not hold, then
there exist indices $i_1,\ldots, i_l$ belonging to $K$ such that $\alpha_{i_1i_2}=1,\ldots,\alpha_{i_li_1}=1$. This 
implies existence of a cycle in $\crit(A,\lambda)$ going through different strongly connected components, contradicting
the fact that they are isolated. 
\end{proof}

\section{One-sided systems and simple image eigenvectors}
\label{s:sis}

\subsection{Solving max-algebraic one-sided systems}

In this section we shall  suppose that  $A\in\Rpmn$ is a given matrix and recall the crucial results concerning  a system of  linear equations $A\otimes x=b$. 
Our notation is similar to that introduced in \cite{But}, \cite{Zim}. However, unlike for example in \cite{But}, Section 3.1, we do not assume that $b$ has full support (i.e., is positive), or even that every row and column of $A$
contains a nonzero element.  
Denote 
\begin{equation}
\label{e:SAb}
S(A,b)=\{x\in \Rpn\colon A\otimes x=b\}. 
\end{equation}
For any $j\in N$ denote
\begin{equation}
\label{e:x*def}
\begin{split}
\gamma^*_j(A,b)&=\max\{\alpha\in\Rp\cup\{+\infty\}\colon\alpha A_{\cdot j}\leq b\}\\
&=\min_{i\colon a_{ij}\neq 0} b_i a_{ij}^{-1},
\end{split}
\end{equation}
assuming that $0\cdot (+\infty)=(+\infty)\cdot 0=0$, $M=\{1,\dots,m\}$ and
\begin{equation} 
M_j(A,b)=\{i\in M\colon a_{ij}\gamma^*_j(A,b)=b_i\neq 0\}.
\end{equation}
Vector $\gamma^*(A,b)=(\gamma^*_j(A,b))_{j=1}^n$ is closely related to the 
projection: $P_A(b)= A\otimes \gamma^*(A,b)$. 

\if{
the projection $P_A(b)$:
\begin{lemma}
\label{l:projx*}
For any $A\in\Rpnn$ and $b\in\Rpm$ we have $P_A(b)=A\otimes \gamma^*(A,b)$.
\end{lemma}
}\fi

$\gamma^*(A,b)$  can have a $+\infty$ component in general, 
but only in the case when the corresponding column is zero: $A_{\cdot j}={\mathbf 0}$.
\if{
\begin{lemma}
\label{l:x*}
\begin{itemize}
\item  $\gamma_j^*(A,b)=\infty$ if and only if $A_{\cdot j}={\mathbf 0}$;
\item  If $\gamma_j^*(A,b)\in\Rp$ then
\begin{equation}
\label{e:x*props}
\begin{split}
\gamma_j^*(A,b)&=\min_{i\colon a_{ij}\neq 0} b_i a_{ij}^{-1},
a_{kj} \gamma_j^*(A,b)= b_k\Leftrightarrow k\in\arg\min_{i\colon a_{ij}\neq 0} b_i a_{ij}^{-1}.
\end{split}
\end{equation}
\end{itemize}
\end{lemma}

Let us show the following.
}\fi

The following lemma is crucial for the theory of $A\otimes x=b$.

\begin{lemma}[e.g. \cite{But} Theorem 3.1.1]
\label{l:x*2}
$x\in S(A,b)$ if and only if $x\leq \gamma^*(A,b)$ and 
$\cup_{j\colon x_j=\gamma_j^*(A,b)} M_j(A,b)=\supp(b)$.
\end{lemma}
\if{
\begin{proof}
The inequality $x\leq \gamma^*(A,b)$ follows from~\eqref{e:x*def}. To prove the rest, assume for the contrary that
there is $i$ such that $b_i>0$ and $i\notin M_j(A,b)$ for all $j\colon x_j=\gamma_j^*(A,b)$. 
By Lemma~\ref{l:x*} $a_{ij}\gamma_j^*(A,b)<b_i$ for all $j\colon x_j=\gamma_j^*(A,b)$, but we also have
$a_{ij}x_j<b_i$ for all $j\colon x_j<\gamma_j^*(A,b)$. But this contradicts $A\otimes x=b$.
\end{proof}
}\fi

We now give a description of the solution set of $A\otimes x=b$ in terms of minimal coverings.

\begin{theorem}
\label{t:ax=b}
Let $\Omega$ be a collection of minimal subsets $N'\subseteq N$ such that
$\bigcup_{j\in N'} M_j(A,b)=\supp(b)$. Then 
\begin{equation}
\label{e:S-SN}
S(A,b)=\bigcup_{N'\in\Omega} S_{N'}(A,b),
\end{equation}
where
\begin{equation}
\label{e:Nprime}
\begin{split}
S_{N'}(A,b)&=\{x\in\Rpn \colon x_j=\gamma_j^*(A,b)\;\text{for}\; j\in N',\\ 
x_j&\leq \gamma_j^*(A,b)\;\text{for}\; j\in N\backslash N'\}.
\end{split}
\end{equation}
\end{theorem}
\begin{proof}
We first show that $S(A,b)\subseteq \bigcup_{N'\subseteq\Omega} S_{N'}(A,b)$. If $x\in S(A,b)$
then by Lemma~\ref{l:x*2} we have $x\leq \gamma^*(A,b)$ and $\bigcup_{j\colon x_j=\gamma_j^*(A,b)} M_j(A,b)=\supp(b)$. Hence 
$x\in S_{N'}(A,b)$ for some minimal $N'\subseteq\{j\colon x_j=\gamma_j^*(A,b)\}$ such that 
$\cup_{j\in N'} M_j(A,b)=\supp(b)$.

Let $x\in S_{N'}(A,b)$. Then $\bigoplus_{j\in N'} A_{\cdot j} x_j=b$ and 
$\bigoplus_{j\in N\backslash N'} A_{\cdot j} x_j\leq b$, hence $A\otimes x=b$.
\end{proof}


\begin{corollary}
\label{c:ax=b}
Let us have $\cup_{j\in N'} M_j(A,b)=\supp(b)$ for some proper $N'\subseteq N$,
and let $i\notin N\backslash N'$. Then for each $\alpha\leq \gamma_i^*(A,b)$ there 
exists $x\in S_{N'}(A,b)\subseteq S(A,b)$ with $x_i=\alpha$. 
\end{corollary}
\begin{proof}
We can assume without loss of generality that $N'$ is a minimal subset such that 
$\cup_{j\in N'} M_j(A,b)=\supp(b)$. The claim then follows from~\eqref{e:Nprime} and~\eqref{e:S-SN}.
\end{proof}

Let us now formulate, without proof, conditions for existence and uniqueness of 
a finite solution to $A\otimes x=b$. Here $\overline{S}(A,b):=\{x\in (\R\cup\{+\infty\})^n\colon A\otimes x=b\}$.

\begin{proposition}[e.g., \cite{But}, Coro. 3.1.2]
\label{p:existence}
Let  $A\in \Rpmn$ and $b\in \Rpm$. Then the following conditions are equivalent:
\begin{enumerate}
\item[{\rm (i)}]
$S(A,b)\neq \emptyset$,
\item[{\rm (ii)}]
$\gamma^*(A,b)\in \overline{S}(A,b)$, 
\item[{\rm (iii)}]
$\bigcup\limits_{j\in N}M_j(A,b)=\supp(b).$
\end{enumerate}
\end{proposition}

\begin{proposition}[e.g., \cite{But}, Coro. 3.1.3]
\label{p:uniqueness}
Let  $A\in \Rpmn$ and $b\in \Rpm$, and let the
solution to $A\otimes x=b$ exist. Then the following are equivalent:
\begin{itemize}
\item[{\rm (i)}] $A\otimes x=b$ has unique solution;
\item[{\rm (ii)}] $S(A,b)=\{\gamma^*(A,b)\}$;
\item[{\rm (iii)}] 
$\cup_{j\neq i} M_j(A,b)\neq\supp(b)\quad\forall i\colon \gamma_i^*(A,b)\neq 0.$
\end{itemize}
\end{proposition}

\subsection{Digraph coverings and systems with eigenvector on the right-hand side}

Let $\digr=(N,E)$ be a strongly connected digraph. Define
\begin{equation}
\label{e:Mjdigr}
M_j(\digr)=\{i\in N\colon (i,j)\in E\}
\end{equation}

\if{
We will need the notions of covering and minimal covering. 
If $D$ is a set and ${\mathcal E}\subseteq {\mathcal P}(D)$ is a set of subsets of $D$,
then ${\mathcal E}$ is said to be a {\em covering} of $D$, if $\bigcup{\mathcal E}=D$ and a covering
${\mathcal E}$ of $D$ is called {\em minimal}, if $\bigcup({\mathcal E}-{F})\neq D$ holds for every $F\in {\mathcal E}$.
}\fi

\begin{theorem}
\label{t:graphcover}
Let $\digr=(N,E)$ be a strongly connected digraph. Then, $i\in N$ with the property
$\cup_{j\in N\backslash\{i\}} M_j(\digr)=N$ exists if and only if $\digr$ is not a cycle.
\end{theorem}
\begin{proof}
Observe first that if $\digr$ is a cycle then there are no nodes with such property.

To prove the converse, observe that if $\digr$ is not a cycle, then it contains two intersecting cycles,
and one of the nodes in the intersection will have at least two ingoing edges.

Take a node with at least two incoming edges, and number this node as $1$. 
Put $N_0=\emptyset$ and $N_1:=\{1\}$. 
For each $l\geq 1$ let 
\begin{equation}
\label{e:Nldef1}
N_{l+1}=N_l\cup \bigcup_{j\in N_l} M_j(\digr).
\end{equation}
Observe that since $\bigcup_{j\in N_{l-1}} M_j(\digr)\subseteq N_l$,
we can replace~\eqref{e:Nldef1} with
\begin{equation}
\label{e:Nldef2}
N_{l+1}=N_l\cup \bigcup_{j\in N_l\backslash N_{l-1}} M_j(\digr).
\end{equation}
Hence for each $l\geq 1$ we have 
\begin{equation}
\label{e:Nl}
N_{l+1}\backslash N_l\subseteq \bigcup_{j\in N_l\backslash N_{l-1}} M_j(\digr).
\end{equation}


Since $\digr$ is finite, there is $t>1$ that satisfies $N_{t+1}=N_t$ and $N_{t-1}\neq N_t$. 
Observe that $N_t=N$, for if $N_t$ is a proper subset of $N$, then this contradicts the
assumption that $\digr$ is strongly connected. 

Consider the case when $|N_t\backslash N_{t-1}|>1$. Observe that a covering of $N$ can be built by taking all nodes in $N_{t-1}$ 
and a node in $N_t\backslash N_{t-1}$ that has an ingoing edge from $1$, 
or just the nodes in 
$N_{t-1}$ if one of the nodes of $N_{t-1}$ has an ingoing edge
from $1$.
Therefore in both of these cases there exists $i\in N_t\backslash N_{t-1}$ with 
$\cup_{j\in N\backslash\{i\}} M_j(\digr)=N$. 

Consider the remaining case when $|N_t\backslash N_{t-1}|=1$ and when the node forming $N_t\backslash N_{t-1}$
is the only node that has an ingoing edge from $1$. 
Since $|N_2\backslash N_1|>1$ there exists 
an $s<t$ with $|N_{s+1}\backslash N_s|<|N_s\backslash N_{s-1}|$. 
Since $N_s\backslash N_{s-1}$
does not contain a node to which $1$ is connected with an edge, for some 
node $i\in N_s\backslash N_{s-1}$ we have
\begin{equation}
\label{e:Nli}
N_{s+1}\backslash N_s=\cup_{j\in N_s\backslash N_{s-1}, j\neq i} M_j(\digr).
\end{equation}
Combining this with~\eqref{e:Nl} for all other $l\neq s$ we obtain that
$$
N=\cup_{j\neq i} M_j(\digr).
$$ 
This completes the proof.
\end{proof}

\begin{corollary}
\label{c:graphcover}
For each $x\in V(A,\lambda)$, an $i$ with the property 
$\cup_{j\neq i} M_j(\crit(A,\lambda,x))=N_c(A,x,\lambda)$ exists if and only if 
$\crit(A,x,\lambda)$ is not a union of disjoint cycles.
\end{corollary}

We now briefly examine the link to 
one-sided systems with eigenvector on the right-hand side.

\begin{proposition}
\label{p:nicelemma}
Let $x\in V(A,\lambda)$, and let $j\in N$ be such that there exists 
$l$ with $(l,j)\in\Sat(A,x,\lambda)$. Then 
\begin{itemize}
\item[{\rm (i)}] $x\leq \gamma^*(A,\lambda x)$;
\item[{\rm (ii)}] $\gamma_j^*(A,x)=x_j$ and $M_j(A,\lambda x)=
M_j(\Sat(A,x,\lambda))$;
\item[{\rm (iii)}] 
$M_j(\crit(A,\lambda))\subseteq M_j(A,\lambda x)$.
\end{itemize}
\end{proposition}
\begin{proof}
(i): follows by Lemma~\ref{l:x*2}.
(ii): Since $x\in V(A,\lambda)$, we have 
$a_{kj}x_j\leq \lambda x_k$ for all $k$, so
$\gamma_j^*(A,\lambda x)=\min_{i\in N}\{\lambda x_i(a_{ij})^{-1}\}=\lambda x_l(a_{lj})^{-1}=x_j$.\\
$M_j(A,\lambda x)=\{i\colon a_{ij}x_j=\lambda x_i\}$ follows from the definition of these sets,
substituting $\gamma_j^*(A,\lambda x)=x_j$.

(iii): By (ii) we have $M_j(A,\lambda x)=M_j(\Sat(A,x,\lambda)$, and
also $$M_j(\crit(A,x,\lambda))\subseteq M_j(\Sat(A,x,\lambda))=M_j(A,\lambda x)$$
since $\crit(A,x,\lambda)\subseteq \Sat(A,x,\lambda)$ by Proposition~\ref{p:satgraph} part (iii). 
Graph $\crit(A,x,\lambda)$ consists of entire components of
$\crit(A,\lambda)$ (Proposition~\ref{p:critAxl} part (iii)), hence 
$$M_j(\crit(A,\lambda))=M_j(\crit(A,x,\lambda))
\subseteq M_j(\Sat(A,x,\lambda))=M_j(A,\lambda x).$$ 
This concludes the proof.
\end{proof}

\subsection{Simple image eigenvectors}

Denote by $A^{(i)}$ the matrix which remains after the $i$th column of $A$
is removed.

\begin{lemma}[e.g. \cite{But} Theorems 3.1.5 and 3.1.6]
\label{l:sis}
Let $A\in\Rpmn$ and $b\in\Rpm$.
Then, $b$ belongs to the simple image set of $A$ 
if and only if there is no $i$ for which $\gamma_i^*(A,b)\neq 0$ and
$b\in \spann(A^{(i)})$.
\end{lemma}
\begin{proof}
By Proposition~\ref{p:existence} $b\in\spann(A^{(i)})$ if and only if
$\cup_{j\in N\backslash\{i\}} M_j(A,b)=\supp(b)$. 
By Proposition~\ref{p:uniqueness} the non-uniqueness of solution to 
$A\otimes x=b$ is equivalent to the existence of $i$ with
$\cup_{j\in N\backslash\{i\}} M_j(A,b)=\supp(b)$ and 
$\gamma_i^*(A,b)\neq 0$.
\end{proof}

\if{
{\color{blue}{In the cited reference, $b$ is assumed not to have any zero components, so that
$\gamma_i^*(A,b)>0$ for all $i$. If $b$ has $0$ components then we may have 
$\gamma_i^*(A,b)=0$ for some $i$, in which case 
the simple image of $A$ is the same as the simple image of the matrix resulting from
deleting such columns from $A$. Observe that the rows of $A$ corresponding to the zero entries of $b$
are also essential, and deleting these rows from $A$ reduces Lemma~\ref{l:sis} to 
\cite{But} Theorems 3.1.5 and 3.1.6}}
}\fi

Thus, $x$ is not a simple image eigenvector if and only if there exists an $i$ (and $\lambda$) 
such that $x\in \spann(A^{(i)})\cap V(A,\lambda)$ and $\gamma_i^*(A,\lambda x)>0$.

\begin{theorem}
\label{t:peter}
$V(A,\lambda)$ contains simple image vectors if and only if there exists a subset $N'\subseteq N$
with the following properties:
\begin{itemize}
\item[{\rm (i)}] There exists $x\in V(A,\lambda)$ with $\supp(x)=N'$;
\item[{\rm (ii)}] $\gamma_i^*(A,x)=0$ for all $i\notin N'$;
\item[{\rm (iii)}] $N'= N_c(A,x,\lambda)$;
\item[{\rm (iv)}] 
All components of $\crit(A,x,\lambda)=\crit(A,\lambda)|_{N'}$ 
are cycles.
\end{itemize}
\end{theorem}
\begin{proof} By Proposition~\ref{p:generate} we can assume without loss of generality
that $\lambda=1$.

``Only if'': Let us first argue that (i) and (ii) are necessary. Let $x$ be a simple image vector in $V(A,\lambda)$, and 
take $N'=\supp(x)$. Condition (i) is immediate, and for (ii) we observe that
that $\gamma_i^*(A,x)>0$ for some $i\notin N'$ would imply that $x$ is not the only solution of the
system $A\otimes y=x$, since  $\gamma^*(A,x)\neq x$. 

As condition (ii) depends only on the support of $x$ and on the zero-nonzero pattern of $A$, it also 
holds for {\em arbitrary} $x\in V(A,1)$ with $\supp(x)=N'$. That condition also implies 
\begin{equation}
\label{e:supporty}
\supp(y)\subseteq N'\ \text{for all}\ y\ \text{such that}\  A\otimes y=x. 
\end{equation}

Moreover, we can further consider the system $A_{N'N'}\otimes y_{N'}=x_{N'}$
with $x_{N'}\in V(A_{N'N'})$, since we have the following correspondence: 
\begin{equation}
\label{e:corr-simple}
\begin{split}
x_{N'}\in V^+(A_{N'N'},1)&\leftrightarrow (x\in V(A,1)\ \&\ \supp(x)=N'),\\
y_{N'}\in S(A,x_{N'})&\leftrightarrow y\in S(A,x).
\end{split}
\end{equation}
In that correspondence, $x$ and $y$ arise from $x_{N'}$ and $y_{N'}$ by setting
$x_{N\backslash N'}=y_{N\backslash N'}={\mathbf 0}$, and $x_{N'}$ and $y_{N'}$ are formed
as the usual subvectors of $x$ and $y$  with indices in $N'$. The first part of that correspondence 
follows from~Proposition~\ref{p:support}, and the second part follows from~\eqref{e:supporty}.  

Therefore, to show that (iii) and (iv) hold we can
further assume that $x$ is positive, that is, $N'=N$. 

Assume for the contrary that (iii) does not hold. 
If $N_c(A,1)\neq N$, consider the indices in $N\backslash N_c(A,1)$. Every node with index in 
$N\backslash N_c(A,1)$ has an outgoing edge in $\Sat(A,x,1)$. It cannot be that all of these 
edges also end in $N\backslash N_c(A,1)$, because then there would be critical cycles and hence critical
components in $N\backslash N_c(A,1)$, a contradiction. Therefore there exists $(i,j)\in E_{\Sat}(A,x,1)$ with $i\notin N_c(A,1)$ and $j\in N_c(A,1)$ implying that
\begin{equation}
\label{e:nci}
N_c(A,1)\cup\{i\}\subseteq \bigcup_{j\in N_c(A,1)} M_j(\Sat(A,x,1)).
\end{equation}
For any $i'\notin N_c(A,1)$ there exists $j'$ such that $i'\in M_j(\Sat(A,x,1)$, and therefore 
$$N_c(A,1)\cup\{i\}\cup\{i'\}\subseteq \bigcup_{j\in N_c(A,1)\cup \{j'\}} M_j(\Sat(A,x,1)).$$ 
Thus adding indices to the left hand side of~\eqref{e:nci} one by one, we obtain that there exists $k\notin N_c(A,1)$ 
such that $N=\cup_{j\neq k}  M_j(\Sat(A,x,1))$, i.e., 
$N=\cup_{j\neq k} M_j(A,x)$, a contradiction.


Now assume that (iii) holds but (iv) does not hold, i.e., one of the components of $\crit(A)$ is not 
a cycle. In this case by Corollary~\ref{c:graphcover} there exists $i$ such that 
$\cup_{j\neq i} M_j(\crit(A,x,1))=N$. However, as $\crit(A,x,1)\subseteq\Sat(A,x,1)$ by 
Proposition~\ref{p:satgraph}, we also have
$M_j(\crit(A,x,1))\subseteq M_j(A,x)$ for all $j\in N$, implying $\cup_{j\neq i} M_j(A,x)=N$, a contradiction. 

``If'': Let us now prove that (i)-(iv) are sufficient. Due to bijection~\eqref{e:corr-simple}
we can assume that $N'=N$. Then, by the main result of~\cite{SSB}, 
there exists a diagonal matrix $X$ such that $\Tilde{A}:=X^{-1}AX$ is strictly visualised,
that is $\Tilde{a}_{ij}\leq 1$, with the equality $\tilde{a}_{ij}=1$ if and only if 
$(i,j)$ is a critical edge, that is, if and only if $(i,j)$ belongs to one of the disjoint critical cycles.
Let $u$ be the vector whose every component is $1$.  For this vector we obtain
$\Sat(\Tilde{A},u,1)=\crit(\Tilde{A},u,1)=\crit(\Tilde{A},1)$ and hence $M_j(\Tilde{A},u)=M_j(\crit(\Tilde{A},1))$ for all $j\in N$. 
Since $\crit(A)$ consists of disjoint cycles only, by Corollary~\ref{c:graphcover} we have 
$\cup_{j\neq i} M_j(\crit(\Tilde{A},1))\neq N$ for any $i\in N$, and since 
$M_j(\Tilde{A},u)=M_j(\crit(\Tilde{A},1))$ for all $j\in N$ we obtain
that $u$ is a simple image eigenvector of $\Tilde{A}$. By Lemma~\ref{l:scale-siseig},
$Xu$ is a simple image eigenvector of $A$. 
\end{proof}

\begin{remark}
\label{r:peter}
{\rm Condition (i) of Theorem~\ref{t:peter} can be expressed in terms of the Frobenius normal form of $A$, see
for instance~\cite{Gau:92} Ch. IV Theorem 2.2.4.
Condition (ii) of Theorem~\ref{t:peter} is equivalent to the following:}
\begin{equation}
\forall i\notin N' \exists l\notin N'\colon a_{li}\neq 0.
\end{equation}
\end{remark}

\if{
\begin{proposition}
\label{p:case1}
If $\cup_{j\neq i} M_j^c(A)=M$  then $W^{(i)}=V(A,\lambda)$.
\end{proposition}
\begin{proof}
Since $M_j^c(A)\subseteq M_j(A,x)$ for any $x\in V(A)$ and $j\in N$, we obtain 
$\cup_{j\neq i} M_j(A,x)=M$ for all $x\in V(A)$ implying that $V(A)\subseteq \spann(A^{(i)}$.
\end{proof}
}\fi

The next result is mostly needed for Theorem~\ref{t:openclosed}, but it is also of independent 
interest. We formulate it for positive vectors only, for the sake of simplicity.

\begin{theorem}
\label{t:case2}
Let $N_c(A,\lambda)=N$ and let all the s.c.c of $\crit(A,\lambda)$ be cycles and  
$K=\{j_1,\ldots,j_k\}$ be the index set of 
		the columns of $A_{\lambda}^+$ forming $G_{A,\lambda}$. Then
\begin{equation}
(\bigcup_{i\in N} \spann(A^{(i)})\cap V^+(A,\lambda)=(\bigcup_{s\in K} \spann^+(G_{A,\lambda}^{(s)}).
\end{equation}
\end{theorem}
\begin{proof} Assume $\lambda=1$. 

For any $x\in V(A,1)$ we have $x\in\spann(A)$ and $x\in\spann(G_{A,1})$, and if $x$ is positive
then we also have coverings $\cup_{i\in N} M_i(A,x)=N$ and $\cup_{r\in K} M_r(G_{A,1},x)=N$. The claim
of the theorem, reducing to $\exists i\colon x\in\spann(A^{(i)})\Leftrightarrow
\exists s\colon x\in\spann(G_{A,1}^{(s)})$ for positive $x\in V(A,1)$, then amounts to equivalence between the following statements:\\
(a) $\exists i\in N\colon M_i(A,x)\subseteq \bigcup_{j\neq i} M_j(A,x)$ and\\
(b) $\exists s\in K\colon M_s(G_{A,1},x)\subseteq \bigcup_{r\neq s} M_r(G_{A,1},x)$.\\
Let $X=\diag(x)$ and consider the matrix
$\Tilde{A}:=X^{-1} A X$.  
Matrix $G_{\Tilde{A},1}$ then generates the cone $V(\Tilde{A},1)$ which is the
same as $\{y\colon Xy\in V(A,1)\}$, by Lemma~\ref{l:scale-siseig} part (ii). 
Proposition~\ref{p:nicelemma} part (ii) implies that $M_i(A,x)=M_i(\Sat(A,x,1)$ for all $i$, since all nodes are critical and $\crit(A,1)\subseteq\Sat(A,x,1)$.

By Proposition~\ref{p:SSB} $\Tilde{A}=:(\Tilde{a}_{ij})_{i,j=1}^n$ has the property $\Tilde{a}_{ij}\leq 1$  
with the equality if and only if $(i,j)\in\Sat(A,x)$, that is, if and only if $i\in M_j(A,x)$.
In $\Tilde{A}^+$, all columns with indices belonging to the same component
of $\crit(\Tilde{A})$ are equal to each other (Proposition~\ref{p:a+props}).

Denote the entries of $G_{A,1}$ and $G_{\Tilde{A},1}$ by $g_{is}$ and $\Tilde{g}_{is}$ 
respectively.

Assume (a). It means that there exists $i\in N_c^s$ for some $s\in K$ 
such that for each $i'\in M_i(A,x)=M_i(\Sat(A,x))$ there is an index 
$l(i')$ with $l(i')\in N_c^{s'}$ with $s'\neq s$, such that 
$i'\in M_{l(i')}(A,x)=M_{l(i')}(\Sat(A,x))$.
In terms of matrix $\Tilde{A}$ it means that 

\begin{equation}
\label{e:conda}
\Tilde{a}_{i'i}=1\Rightarrow\Tilde{a}_{i'l(i')}=1\ 
\text{with $l(i')\in N_c^{s'}$, $i\in N_c^s$ and $s\neq s'$.}
\end{equation}

We will show that for $s\in K$ as above and for every $e\in N$ such that $\Tilde{g}_{es}=1$ we can find 
$s'$ with $\Tilde{g}_{es'}=1$. Firstly if $\Tilde{g}_{es}=1$ then $\Tilde{a}^+_{ej_s}=1$ and
$\Tilde{a}^+_{ei}=1$ since $i$ and $j_s$ are in the same cycle of $\crit(A)$. This implies that for some $i'$ there is a path $P$ of weight $1$ connecting
$e$ to $i'$, and edge $(i',i)$ of weight $1$, in digraph $\digr(\Tilde{A})$.
By~\eqref{e:conda} there exists index $l(i')$ such that 
$\Tilde{a}_{i'l(i')}=1$ with $l(i')\in N_c^{s'}$ and $s'\neq s$. Concatenating $P$ with edge $(i',l(i'))$ we obtain a path of weight $1$ such that 
$\Tilde{a}^+_{el(i')}=1$ and hence $\Tilde{g}_{es'}=1$ with $s'\neq s$.

Assume (b).  By Proposition~\ref{p:a+props} part (i) we have $\Tilde{g}_{es}=1$ for any $s\in K$ and
$e\in N_c^s$, and statement (b) implies that there also exist $e$ and $r\in K$ such 
that $\Tilde{g}_{er}=1$ and $e\in N_c^s$ with $s\neq r$.  
Moreover, by Proposition~\ref{p:a+props} part (ii) we can also assume that $\Tilde{g}_{is}<1$ for all 
$i\notin N_c^s$.

As $\Tilde{g}_{er}=1$ for all $e\in N_c^s$, we have 
$\Tilde{a}^+_{ej}=1$ for all $e\in N_c^s$ and $j\in N_c^r$. 
This implies that there exist $i_1\in N_c^s$, $i'\in N_c^{s'}$ 
such that $\Tilde{a}_{i_1i'}=1$, where $s\neq s'$. But as $i_1\in N_c^s$,
there also exists $i_2\in N_c^s$ such that $\Tilde{a}_{i_1i_2}=1$.

As $\Tilde{g}_{es}<1$ for all $e\notin N_c^s$, we have
$\tilde{a}^+_{ei_2}<1$ implying 
also that $\tilde{a}_{ei_2}<1$ for all $e\notin N_c^s$. 
Since each component of $\crit(A,1)$ is a cycle, we conclude that $M_{i_2}(\Sat(A,x))=\{i_1\}$, and it is covered by
$M_{i'}(\Sat(A,x))$, implying (a).
\end{proof}

\section{Interval problems}
\label{s:int}

\subsection{Weak $\mbf{X}$-robustness and $\mbf{X}$-simple image eigencone}
In this section we consider an interval extension of weak robustness and
its connection to $\mbf{X}$-simplicity, the main notion studied in this paper.

\begin{definition}[Weak $\mbf{X}$-robustness]
\label{def:weakrob}
 Let $A\in \Rpnn$ and $\mbf{X}\subseteq\Rpn$ be an interval.  
\begin{itemize}
\item[{\rm (i)}] $\Attr(A,\lambda):=\{x\colon A^{\otimes t}\otimes x\in V(A,\lambda)$ for some $t\}$.
\item[{\rm (ii)}] $A$ is called 
weakly $(\mbf{X},\lambda)$-robust  if $\attr(A,\lambda)\cap \mbf{X}\subseteq V(A,\lambda)$.
\end{itemize}
\end{definition}

If $\mbf{X}=\RR^n$, then the notion of weak robustness
can be described in terms of simple image eigenvectors. 
The proof is omitted.

\begin{proposition} \label{Th_WXRminplus}
$A\in\Rpnn$. The following are equivalent:
\begin{enumerate}
\item $A$ is weakly $\lambda$-robust;
\item $|S(A,x)|=1$ for all $x\in V(A,\lambda)$;
\item Each $x\in V(A,\lambda)$ is a simple image eigenvector.
\end{enumerate}
\end{proposition}
\if{
{\it Proof.} We will only prove the equivalence between the
first two claims (the other equivalence being evident). Suppose that there is $x\in V(A)$ such that $|S(A,x)|>1$ (notice that $|S(A,x)|\geq 1$ for each $x$ because of $x\in V(A)$). Then there is at least one solution $y$ of the system $A\otimes y=x$ and $y\neq x$. Using \thmref{WS} we get
$A\otimes(A\otimes y)=A\otimes x=x$ and $A\otimes y=x\neq y$, this is a contradiction.\\
The converse implication trivially follows.\kd
}\fi

\if{
\begin{definition} Let $A \in \RR(n,n)$ and $\mbf{X}$ (closed or half-closed) be given.
\begin{enumerate}
\item
An eigenvector $x\in V(A)\cap\mbf{X}$ is called an
$\mbf{X}$-simple image eigenvector if $x$ is the unique solution of the
equation $A\otimes y=x$ in interval $\mbf{X}$.
\item
Matrix $A$ is said to have $\mbf{X}$-simple image eigencone if any
$x\in   V(A)\cap\mbf{X}$ is an $\mbf{X}$-simple image eigenvector.
\end{enumerate}
\end{definition}
}\fi

\begin{definition}[Invariance]
 Let $A \in \Rpnn$ and $\mbf{X}\subseteq\Rpn$ be an interval.
We say that $\mbf{X}$ is invariant under $A$ if $x\in \mbf{X}$ implies $A\otimes x\in \mbf{X}.$
\end{definition}

\begin{theorem} \label{WXR}
Let   $A\in\Rpnn$  and $\mbf{X}$ be an interval.
\begin{enumerate}
\item
If $A$ is weakly $(\mbf{X},\lambda)$-robust then
$A$ has $\mbf{X}$-simple image eigencone associated with $\lambda$.
\item
If $A$ has $\mbf{X}$-simple image eigencone associated with $\lambda$ and if
$\mbf{X}$ is invariant under $A$ then $A$ is
weakly $(\mbf{X},\lambda)$-robust.
\end{enumerate}
\end{theorem}
\begin{proof}
(i) Suppose  that $A$ is weakly $(\mbf{X},\lambda)$-robust and
$x\in V(A,\lambda)\cap\mbf{X}$. If the system $A\otimes y=\lambda x$ has a solution $y\neq x$ in
$\mbf{X}$, then $y$ is not an eigenvector but belongs to $\attr(A,\lambda)\cap\mbf{X}$,
which contradicts the weak $\mbf{X}$-robustness.

(ii) Assume that $A$ has $\mbf{X}$-simple image eigencone and
$x$ is an arbitrary element of $\attr(A,\lambda)\cap\mbf{X}$.
As $\mbf{X}$ is invariant under $A$, we have that $A^{\otimes k}\otimes x\in\mbf{X}$
for all $k$. Moreover from the definition of 
$\mbf{X}$-simple image eigencone 
we get $A\otimes x\in V(A,\lambda)\Rightarrow x\in V(A,\lambda)$. Then $A^{\otimes k}\otimes x\in
V(A,\lambda)$ for some $k$ implies
$A^{\otimes(k-1)}\otimes x\in V(A,\lambda)$,...,
$x\in V(A,\lambda)$.
\end{proof}

Thus the $\mbf{X}$-simplicity is a necessary condition for
weak $\mbf{X}$-robustness. It is also sufficient if
$\mbf{X}$ is invariant under $A$.

Observe that $A$ is order-preserving $(x\leq y\Rightarrow A\otimes x\leq A\otimes y)$, which is due to
the following arithmetic properties:
\begin{equation*}
\begin{split}
& x\leq y \Rightarrow \alpha x\leq \alpha y\quad \forall\alpha,x,y\in\Rp\\
& \alpha_1\leq \alpha_2,\,\beta_1\leq\beta_2\Rightarrow \alpha_1\oplus\beta_1\leq\alpha_2\oplus\beta_2\quad\forall \alpha_1,\alpha_2,\beta_1,\beta_2\in\Rp.
\end{split}
\end{equation*}

Since $A$ is order-preserving the invariance of $\mbf{X}$ under $A$ admits the following
simple characterization:

\begin{proposition}
\label{p:inv}
Let $\mbf{X}$ be closed. Them $\mbf{X}$ is invariant under $A$ if and only if 
$\underline{x}\leq A\otimes \underline{x}\leq A\otimes \overline{x}\leq\overline{x}$.
\end{proposition}


\if{
We also remark {\color{blue}{that invariance means $A\otimes\mbf{X}\subseteq \mbf{X}$ for $A\otimes\mbf{X} =\{A\otimes x;\,x\in\mbf{X}\}$}} and if $\mbf{X}$ is more specific (say, closed or open), then the conditions $A\otimes\underline{x}\in\mbf{X}$
and $A\otimes\overline{x}\in\mbf{X}$ can be made more precise than in Proposition~\ref{p:inv}. 
}\fi

\subsection{$\mbf{X}$-simple image of a matrix}

Let us first introduce the following bits of notation:
\begin{definition}
\label{def:oc}
\begin{equation}
c(\underline{x})=\{i\in N\colon\underline{x}_i\in\mbf{X}_i\},\quad o(\underline{x})=N\backslash c(\underline{x})
\end{equation}
\end{definition}
\begin{definition}
\label{def:uparrow}
\begin{equation}
\begin{split}
\mbf{X}^{\uparrow}&=\times_{i=1}^n \mbf{X}_i^{\uparrow},\ \text{where}\\
\mbf{X}_i^{\uparrow}&=\mbf{X}_i\cup[\sup(\mbf{X}_i),+\infty). 
\end{split}
\end{equation}
\end{definition}
We illustrate the definitions by the following example.
\begin{example} 
Suppose that $\mbf{X}=[1,2]\times [3,5)\times (7,9]$ be given. Then we have
$\underline{x}=(1,3,7),\,c(\underline{x})=(1,2)$ and
$\mbf{X}^{\uparrow}=[1,\infty)\times [3,\infty)\times (7,\infty).$
\end{example}
\begin{lemma}
\label{l:X-exist}
Let $A\in\Rpmn$ and $b\in\Rpm$.
A solution to $A\otimes x=b$ in $\mbf{X}$ exists if and only if
\begin{equation}
\label{e:X-exist}
\gamma^*(A,b)\in\mbf{X}^{\uparrow}\ \& \bigcup_{j\colon \gamma_j^*(A,b)\in\mbf{X}_j} M_j(A,b)=\supp(b).
\end{equation}
\end{lemma}
\begin{proof}
``Only if''. Assume  that $\mbf{X}$ contains a solution to $A\otimes x=b$, but~\eqref{e:X-exist} does not hold. If $\gamma^*(A,b)\notin\mbf{X}^{\uparrow}$, then since every solution to 
$A\otimes x=b$ satisfies $x\leq \gamma^*(A,b)$, we have $x\notin \mbf{X}^{\uparrow}$ and 
hence $x\notin\mbf{X}$ for every solution.  Let $\gamma^*(A,b)\in\mbf{X}^{\uparrow}$, 
$\bigcup_{j\colon \gamma_j^*(A,b)\in\mbf{X}_j} M_j(A,b)\neq \supp(b)$. Since $x\in\mbf{X}$ we have $x_j<\gamma^*_j(A,b)$ for all $j$ such that $\gamma^*_j(A,b)\notin\mbf{X}_j$. This implies that $A\otimes x\neq b$, a contradiction.

``If''. Assume that~\eqref{e:X-exist} holds. 
Then by Corollary~\ref{c:ax=b} any $x$ with
$x_j=\gamma_j^*(A,b)$
for $j\colon \gamma_j^*(A,b)\in \mbf{X}_j$, and
$x_j\leq \gamma_j^*(A,b)$ 
for $j\colon \gamma_j^*(A,b)\in \mbf{X}_j^{\uparrow}\backslash \mbf{X}_j$
is a solution to $A\otimes x=b$. Hence $A\otimes x=b$ has a solution in $\mbf{X}$.
\end{proof}

\begin{theorem}
\label{t:X-unique}
Let $A\otimes x=b$ have a solution in $\mbf{X}$. That solution is unique in 
$\mbf{X}$ if and only if the following condition is satisfied for each $i$:
\begin{equation}
\label{e:X-unique}
\bigcup_{j\neq i} M_j(A,b)=\supp(b)\Rightarrow [i\in c(\underline{x})\ \&\ \underline{x}_i=\min(\overline{x}_i,\gamma_i^*(A,b))]
\end{equation}
\end{theorem}
\begin{proof}
As $A\otimes x=b$ has a solution in $\mbf{X}$, condition~\eqref{e:X-exist} holds, and by Corollary~\ref{c:ax=b}
we have
\begin{equation}
\label{e:alpha1}
\forall \alpha\in \mbf{X}_i\  \text{s.t.}\  \alpha\leq \gamma_i^*(A,b)\ \exists x\in S(A,b)\cap\mbf{X}\colon x_i=\alpha.
\end{equation}

In particular if $\gamma_i^*(A,b)\in\mbf{X_i}^{\uparrow}\backslash\mbf{X}_i$ then
\begin{equation}
\label{e:alpha}
\forall \alpha\in \mbf{X}_i\  \exists x\in S(A,b)\cap\mbf{X}\colon x_i=\alpha.
\end{equation}

``Only if'':  Suppose that there is a unique solution belonging to $\mbf{X}$. 
By~\eqref{e:alpha} if $\mbf{X}_i$ does not reduce to one point, then $S(A,b)$ also contains more than one vector.
Thus, $\underline{x}_i=\overline{x}_i$ for all $i$ such that $\gamma_i^*(A,b)\in\mbf{X_i}^{\uparrow}\backslash\mbf{X}_i$, which satisfies~\eqref{e:X-unique}.


If we assume that~\eqref{e:X-unique} does not hold for some $i$ with 
$\gamma_i^*(A,b)\in\mbf{X_i}$ then 
$\bigcup_{j\neq i} M_j(A,b)=\supp b$ 
and~\eqref{e:alpha} implies that the solution is non-unique since the interval 
$\mbf{X}_i\cap\{\alpha\colon \alpha\leq \gamma_i^*(A,b)\}$ contains more than one point.

``If'':  Assume that~\eqref{e:X-unique} holds and, by contradiction,
that there is more than one solution to $A\otimes x=b$ belonging to $\mbf{X}$.

Since the solution is non-unique, it follows that there exists a proper subset $N'$ of $N$
such that $\cup_{j\in N'} M_j(A,b)= \supp(b)$. Assume that $N'$ is a minimal such subset, with respect to inclusion. 

We have $\cup_{j\neq i}  M_j(A,b)= \supp(b)$ for all
$i\in N\backslash N'$. By~\eqref{e:X-unique} $N\backslash N'\subseteq c(\underline{x})$, and $\underline{x}_i=\min(\overline{x}_i,\gamma_i^*(A,b))$ for all $i\in N\backslash N'$. This condition implies that there exists only one $N'$-solution to $A\otimes x=b$ belonging to $\mbf{X}$, and it has coordinates
\begin{equation*}
x_i=
\begin{cases}
\underline{x}_i, &\text{if $\underline{x}_i=\overline{x}_i$},\\
\gamma_i^*(A,b), &\text{otherwise}.
\end{cases}
\end{equation*} 

As this solution is the same for any minimal subset $N'$, system $A\otimes x=b$ has a unique solution,
contradicting the non-uniqueness of it.
 \end{proof}

\begin{corollary}
\label{c:si-eig}
Let $x\in V(A)\cap\mbf{X}$. Then $x$ is an $\mbf{X}$-simple image eigenvector
if and only if
\begin{equation}
\label{e:si-eig1}
\bigcup_{j\neq i} M_j(A,b)=\supp(x)\Rightarrow [i\in c(\underline{x})\ \&\ \underline{x}_i=\min(\overline{x}_i,\gamma_i^*(A,b))]
\end{equation}
If $N_c=N$ then this condition can be replaced with the following one:
\begin{equation}
\label{e:si-eig2}
\bigcup_{j\neq i} M_j(A,b)=\supp(x)\Rightarrow [i\in c(\underline{x})\ \&\ \underline{x}_i=\min(\overline{x}_i,x_i)]
\end{equation}
\end{corollary}
\begin{proof}
As $x\in V(A)\cap\mbf{X}$, system $A\otimes y=x$ has a solution in $\mbf{X}$, which is $y=x$.
So we can apply Theorem~\ref{t:X-unique} yielding~\eqref{e:si-eig1} for the 
uniqueness of this solution. Further if all nodes are critical, then $\gamma_i^*(A,x)=x_i$ for all 
$x$, so~\eqref{e:si-eig1} gets replaced with~\eqref{e:si-eig2}.
\end{proof}

\subsection{Characterizing matrices with $\mbf{X}$-simple image eigencone}

We begin with the following definition and key lemma of geometric kind.

\begin{definition}
	\label{def:openclosed}
	An interval $\mbf{X}\subseteq\Rpn$ is called $\underline{x}$-open if $o(\underline{x})=N$. 
	It is called $\overline{x}$-closed if it $\overline{x}\in\mbf{X}$.
\end{definition}

\begin{lemma} 
\label{l:geometric}
Let $\mbf{X}$ be an $\overline{x}$-closed interval and let $A\in \R^{m\times n}$. 
Then $\mbf{X}\cap\spann(A)\neq\emptyset$ if and only if 
$P_A \overline{x}\in\mbf{X}$.
\end{lemma}
\begin{proof}
If $P_A \overline{x}\in\mbf{X}$ then $\mbf{X}\cap\spann(A)\neq\emptyset$ (since $P_A\overline{x}\in\spann(A)$).


If $\mbf{X}\cap\spann(A)\neq \emptyset$, take $y\in \mbf{X}\cap\spann(A)\neq \emptyset$.
Vector $z=y\oplus P_A \overline{x}$ belongs to $\spann(A)$ and satisfies $y\leq z\leq\overline{x}$.
It follows then that $z\in\mbf{X}$. However, $P_A\overline{x}\leq z\leq\overline{x}$
while $P_A\overline{x}$ is the greatest vector of $\spann(A)$ bounded from above by $\overline{x}$.
This implies that $z=P_A\overline{x}$ and $P_A\overline{x}\in\mbf{X}$.
\end{proof}
\begin{definition}
\label{def:xli}
For any $l,i\in N$, let
\begin{equation}
\label{e:xl}
\overline{x}^{\langle l\rangle}_i=
\begin{cases}
\underline{x}_l, &\text{if $i=l$},\\
\overline{x}_i, &\text{otherwise},
\end{cases}
\end{equation}
and let $\overline{x}^{\langle l\rangle}=(\overline{x}^{\langle l\rangle}_i)_{i\in N}$.
\end{definition}
\begin{definition}
\label{def:ints}
For any $l\in N$, let
\begin{equation}
\label{e:XAl}
\begin{split}
\mbf{X}_{A,\lambda}^l&=\{x\in\mbf{X}\colon x_l=\underline{x}_l,\quad 
\lambda x_j>a_{jl}\underline{x}_l\;\forall j\},\\
\mbf{X}^{(l}&=\{x\in\mbf{X}\colon x_l>\underline{x}_l\}.
\end{split}
\end{equation}
\end{definition}

Observe that if $\mbf{X}$ is $\overline{x}$-closed then $\mbf{X}_A^l$ is $\overline{x}^{\langle l\rangle}$-closed
(if it is nonempty) for every $l$. Also note that $a_{ll}<\lambda$ is a necessary condition for
$\mbf{X}_{A,\lambda}^l$ to be non-empty.

\begin{theorem}
\label{t:mainres}
Let $\lambda>0$ and $V(A,\lambda)\cap\mbf{X}\neq\emptyset$.
\begin{itemize}
\item[{\rm (i)}] $A$ has $\mbf{X}$-simple image eigencone corresponding to $\lambda$ if and only if 
\begin{equation}
\label{e:mainres1}
\begin{split}
& \spann(A^{(i)})\cap V(A,\lambda)\cap\mbf{X}^{(i}=\emptyset\quad\forall i,\\
& V(A,\lambda)\cap\mbf{X}_{A,\lambda}^l=\emptyset\quad\forall l\in c(\underline{x})\backslash N_c(A,\lambda)\ 
\text{s.t.}\ \underline{x}_l<\overline{x}_l.
\end{split}
\end{equation}
\item[{\rm (ii)}] If $\mbf{X}$ is $\overline{x}$-closed then~\eqref{e:mainres1} is equivalent to
\begin{equation}
\label{e:mainres2}
\begin{split}
& P_{\cW^i} \overline{x}\notin\mbf{X}^{(i}\ \forall i\in N\;\text{where}\;
 \cW^{(i)}=\spann(A^{(i)})\cap V(A,\lambda),\\
& P_{V(A,\lambda)} \overline{x}^{\langle l\rangle}\notin \mbf{X}_{A,\lambda}^l,\;\forall l\in c(\underline{x})
\backslash N_c(A,\lambda)\;\text{s.t.}\; \underline{x}_l<\overline{x}_l.
\end{split}
\end{equation}
\end{itemize}
\end{theorem}
\begin{proof} Assume without loss of generality that $\lambda=1$.

(i): Let $x\in V(A,1)\cap\mbf{X}$. In general, $x\leq \gamma^*(A, x)$. More precisely,
for $\gamma_i^*(A, x)$ we may have $\gamma_i^*(A, x)=x_i$ (for all $i\in N_c(A,1)$ and some other nodes),
or $\gamma_i^*(A, x)>x_i$ (for at least one node in $N\backslash N_c(A,1)$ if 
$N\neq N_c(A,1)$).

``Only if'': Let us show that the conditions~\eqref{e:mainres1} are necessary. For this assume by contradiction
that one of these conditions is violated but $A$ has $\mbf{X}$-simple image eigencone.

(a) If the first condition does not hold then take $x'\in\mbf{X}^{(i}\cap V(A,1)\cap \spann(A^{(i)})$. Then $x'$ 
satisfies $x'_i>\underline{x}_i$ and also
$\cup_{j\neq i} M_j(A,x')=\supp(x')$, hence
$x'$ is not an $\mbf{X}$-simple image eigenvector by Corollary~\ref{c:si-eig}: either 
$i\in o(\underline{x})$ and $\cup_{j\neq i} M_j(A,x')=\supp(x')$, or $i\in c(\underline{x})$,
$\cup_{j\neq i} M_j(A,x')=\supp(x')$ and $\underline{x}_i<\min (\gamma_i^*(A,x'), \overline{x}_i)$.

(b) If the second condition does not hold then take 
$x''\in V(A,1)\cap\mbf{X}_A^l$ for some $l\in c(\underline{x})\backslash N_c(A,1)$ 
such that $\underline{x}_l<\overline{x}_l$. We have $x''_l=\underline{x}_l,$ $\underline{x}_l<\overline{x}_l$ and 
$\gamma_l^*(A,x'')>x''_l$ (equivalent with the condition 
$x''_j>a_{jl} x''_l$ $\forall j$ from~\eqref{e:XAl}). 
The inequality $\gamma_l^*(A,x'')>x''_l$ implies that $\cup_{j\neq l} M_j(A,x)=\supp(x'')$, 
and we also have $\underline{x}_l<\min (\gamma_l^*(A,x''), \overline{x}_l)$ and $l\in c(\underline{x})$.
By Corollary~\ref{c:si-eig}, this shows that $x$ is not an $\mbf{X}$-simple image eigenvector.

``If'':  By contradiction, suppose that the conditions hold but $A$ does not have simple image eigencone. Let
$x$ be an $\mbf{X}$-simple image eigenvector. Then either 
$x_i>\underline{x}_i$ and $\cup_{j\neq i} M_j(A,x)=\supp(x)$ for some $i$, which 
implies $x\in\mbf{X}^{(i}\cap V(A)\cap \spann(A^{(i)})$, or 
$x_l=\underline{x}_l$, $\cup_{j\neq l} M_j(A,x)=\supp(x)$ and $\underline{x}_l<\min (\gamma_l^*(A,x), \overline{x}_l)$  for some $l\in c(\underline{x})$. In this case necessarily $\underline{x}_l<\overline{x}_l$ and $x_l<\gamma_l^*(A,x)$, which 
is only possible for $l\in N\backslash N_c$.
 
This shows the sufficiency of~\eqref{e:mainres1}.

(ii) By Lemma~\ref{l:geometric}, $\cW^i\cap\mbf{X}^{(i}\neq\emptyset$ 
if and only if $P_{\cW^i}\overline{x}\in\mbf{X}^{(i}$, and $V(A,1)\cap\mbf{X}_A^{l}\neq \emptyset$
if and only if $P_{V(A,1)}\overline{x}^{\langle l\rangle}\in \mbf{X}_A^{l}$.
\end{proof}


In general, the basis of $\spann(A^{(i)})\cap V(A,\lambda)$ can be computed algorithmically, using the 
method of Butkovi\v{c}, Hegedus~\cite{BH} or the more recent and efficient methods of Allamigeon et al.~\cite{All+}

Let us now examine the case when $\mbf{X}$ is $\underline{x}$-open.

\begin{theorem}
\label{t:openclosed}
Let $\mbf{X}$ be an $\underline{x}$-open interval and $V(A,\lambda)\cap\mbf{X}$ be non-empty.
\begin{itemize}
\item[{\rm (i)}] $A$ has $\mbf{X}$-simple image eigencone corresponding to $\lambda$ 
if and only if
\begin{equation}
\label{e:openclosed1}
\spann(A^{(i)})\cap V(A,\lambda)\cap \mbf{X}=\emptyset\quad\forall i=1,\ldots,n.
\end{equation}
\item[{\rm (ii)}] In that case $N_c(A,\lambda)=N$ and 
$\crit(A,\lambda)$ consists of disjoint cycles $c_1,\ldots,c_k$ for some $k$.
\item[{\rm (iii)}] Condition~\eqref{e:openclosed1} is equivalent to
\begin{equation}
\label{e:openclosed2}
\spann(G_{A^{(s)},\lambda})\cap\mbf{X}=\emptyset\quad\forall s=1,\ldots,k.
\end{equation}
\item[{\rm (iv)}] If $\mbf{X}$ is also $\overline{x}$-closed then~\eqref{e:openclosed1} 
is also equivalent to
\begin{equation}
\label{e:openclosed4}
P_{G_{A^{(s)},\lambda}} \overline{x}\not>\underline{x}\quad\forall s=1,\ldots,k.
\end{equation}

\end{itemize}
\end{theorem}
\begin{proof} Assume $\lambda=1$.
(i): Follows from Theorem~\ref{t:mainres} part (i), where we take into account that 
$c(\underline{x})$ is empty and $\mbf{X}^{(i}=\mbf{X}$ for each $i$.\\

(ii): Note that each vector in $V(A,1)\cap\mbf{X}$ is positive. By Theorem~\ref{t:peter},
if $\crit(A)$ does not consist of disjoint cycles or $N_c\neq N$ then each vector $x\in V(A,1)$
belongs to $\spann(A^{(i)})$ for some $i\in N$. Hence
$\spann(A^{(i)})\cap V(A,1)\cap\mbf{X}\neq\emptyset$ for some $i$, a contradiction.\\

(iii): By Theorem~\ref{t:case2} we have 
$$(\bigcup_{i\in N} \spann(A^{(i)})\cap V(A,1)=
\cup_{s=1}^k\spann(G_{A^{(s)},1}).$$ 
Hence~\eqref{e:openclosed1} and~\eqref{e:openclosed2} are equivalent.\\

(iv): By applying Lemma~\ref{l:geometric} to~\eqref{e:openclosed2}, we obtain that~\eqref{e:openclosed2} is equivalent to
\begin{equation}
\label{e:openclosed3}
P_{G_{A^{(s)},1}} \overline{x}\notin\mbf{X}\quad\forall s=1,\ldots,k,
\end{equation}
and that is the same as~\eqref{e:openclosed4}.
\end{proof}



\begin{thebibliography}{99}


\bibitem{All+} X. Allamigeon, \'{E}. Goubault and S.~Gaubert. 
Computing the Vertices of Tropical Polyhedra using Directed Hypergraphs. 
{\em Discrete and Computational Geometry}, 49:2 (2013), 247–-279. 



\bibitem{BCOQ}
F.L.~Baccelli, G.~Cohen, G.J.~Olsder and J.P.~Quadrat. {\em Synchronization and Linearity.}
Wiley and Sons, 1992. Available online: 
\url{https://www.rocq.inria.fr/metalau/cohen/SED/book-online.html}

\bibitem{But}
P. Butkovi\v{c}.
\newblock  {\em Max-linear Systems: Theory and Algorithms}, Springer Monographs in Mathematics, Springer-Verlag 2010. 


\bibitem{ButSIS}
P.~Butkovi\v{c}. Simple image set of $(\max,+)$ linear mappings.
\newblock{\em Discrete Appl. Math.} vol. 105 (2000), 73 -- 86.

\bibitem{BH} P.~Butkovi\v{c} and G.~Heged\"{u}s. An elimination method for finding all solutions of the system of linear equations over an extremal algebra
{\it Ekonomicko-Matematick\'y Obzor}, 20 (1984), pp. 203–214

\bibitem{BSS}
P.~Butkovi\v{c}, H.~Schneider and S.~Sergeev.  Generators, extremals and
bases of max cones. {\em Linear Algebra Appl.} {\bf 421}, 2007, 394-406.

\bibitem{BSSws}
P.~Butkovi\v{c}, H.~Schneider and S.~Sergeev. Recognizing weakly stable matrices.
{\em SIAM J. Control Optim.} {\bf 50} (5), 2012, 3029-3051.

\bibitem{CGQS}
G.~Cohen, S.~Gaubert, J.P.~Quadrat and I.~Singer. 
Max-plus convex sets and functions. 
In: G.L. Litvinov and V.P. Maslov (eds.), {\em Idempotent Mathematics and Mathematical Physics},
Cont. Math. {\bf 377} , AMS, 2005, pp. 105--129,

\bibitem{CG:79} 
R. A. Cuninghame-Green.
\newblock {\em Minimax Algebra.}
 \newblock Lecture Notes in Economics and Mathematical Systems {\bf 166},
Springer, Berlin, 1979.

\bibitem{Gau:92}
S.~Gaubert. 
\newblock {\em Th\'{e}orie des syst\`{e}mes lin\'{e}aires dans les dio\"{\i}des.}
\newblock PhD Thesis, L'{\'E}cole des Mines de Paris, 1992. Available online: 
\url{http://www.cmap.polytechnique.fr/~gaubert/PAPERS/ALL.pdf}

\bibitem{CG:95} 
R. A.~Cuninghame-Green.
\newblock {\em Minimax algebra and applications.}
\newblock Advances in Imaging and Electron Physics {\bf 90},
 (1995)   1--121.

\bibitem{gazi08} 
M.~Gavalec, K.~Zimmermann.
Classification of solutions to systems of two-sided equations with
interval coefficients. Inter. J. of Pure and Applied Math. {\bf 45}
(2008), 533--542.



\bibitem{HOW}
B.~Heidergott, G.-J. Olsder, and J.~van~der Woude.
\newblock {\em Max-plus at Work}.
\newblock Princeton Univ. Press, 2005.






\bibitem{mys05} 
	H.~My\v{s}kov\'{a}.
	\newblock Interval systems of max-separable linear equations.
	\newblock {\itshape Linear Algebra and Its Applications} 403 (2005)   263--272.
	
	
	\bibitem{mys06} 
	H.~My\v{s}kov\'{a}.
	\newblock Control solvability of interval systems of max-separable linear equations.
	\newblock {\itshape Linear Algebra and Its Applications}  416 (2006) 215--223.

\if{
\bibitem{mp}
H. My\v skov\'a, J. Plavka.
\newblock 
 The robustness of interval matrices in max-plus algebra. 
\newblock {\itshape Linear Algebra and Its Applications} 445 (2014) 85--102.
}\fi

\bibitem{p} 
J. Plavka.
\newblock The weak robustness of interval matrices in max-plus algebra
\newblock {\itshape Discrete Appl. Math.}  173 (2014) 92-101.

\bibitem{P1}
	J. Plavka.
	\newblock The weak robustness of interval matrices in max–plus algebra. 
	\newblock {\itshape Discrete Applied Mathematics} 173 (2014) 92--101. 
	

\bibitem{ro} J.~Rohn.
\newblock Systems of Linear Interval Equations.
\newblock {\itshape Linear Algebra and Its Applications} 126 (1989)  39--78.

\if{
\bibitem{Ser-09}
S. Sergeev.
Max algebraic powers of irreducible matrices in the
periodic regime: An application of cyclic classes, {\itshape Linear Algebra and its Applications} 431 (2009) 1325–1339.
}\fi



\bibitem{Ser-11}
S. Sergeev. Max-algebraic cones of nonnegative irreducible matrices, {\itshape Linear Algebra and its Applications} 435 (2011), 1736-1757. 


\bibitem{SSB}
S. Sergeev, H. Schneider and P. Butkovi\v{c}. On visualization scaling,
subeigenvectors and Kleene stars in max algebra. {\em Linear
Algebra and Its Applications}~431 (2009) 2395--2406.


\bibitem{Vor} 
N.N. Vorobyev. Extremal algebra of positive matrices (in Russian). 
{\em Elektronische Informationsverarbeitung und Kybernetik}, 3 (1967) 39--71.

\bibitem{Zim}
K. Zimmermann. {\em Extrem\'aln\' i algebra} (in  Czech), Ekon.
\'ustav \v{C}SAV Praha, 1976.



\end{thebibliography}
\end{document}